\documentclass[11pt,reqno]{amsart}

\usepackage{amsmath,amssymb,amsthm}
\usepackage[margin=28mm]{geometry}
\usepackage{hyperref}
\usepackage{enumitem}

\numberwithin{equation}{section}
\newtheorem{theorem}{Theorem}[section]
\newtheorem{lemma}[theorem]{Lemma}
\newtheorem{proposition}[theorem]{Proposition}
\newtheorem{corollary}[theorem]{Corollary}
\theoremstyle{remark}
\newtheorem{remark}[theorem]{Remark}

\begin{document}
\allowdisplaybreaks

\title{A Sharp Gradient Harnack Estimate for the Lane--Emden System}

\author{Sanghoon Lee}
\address[Sanghoon Lee]{Research Institute of Mathematics,
Seoul National University, 1 Gwanak-ro, Gwanak-gu,
Seoul 08826, Republic of Korea}
\email{lsh514@snu.ac.kr}

\author{Kui Li}
\address[Kui Li]{School of Mathematics and Statistics,
Zhengzhou University, Zhengzhou 450001, China}
\email{likui@zzu.edu.cn}

\author{Juncheng Wei}
\address[Juncheng Wei]{Department of Mathematics,
The Chinese University of Hong Kong,
Shatin, New Territories, Hong Kong}
\email{wei@math.cuhk.edu.hk}

\date{\today}



\begin{abstract}
We establish a sharp pointwise gradient estimate for positive entire
classical solutions of the Lane--Emden system when one of the exponents
lies below the Serrin exponent $n/(n-2)$.
The estimate generalizes the positivity of the scalar curvature
of the associated conformal metric in the biharmonic case $q=1$.
The coefficient in the estimate is optimal on the critical hyperbola.
\end{abstract}

\maketitle

\section{Introduction}
The Lane--Emden system
\begin{equation}\label{eq:lane-emden}
-\Delta u=v^p,\qquad -\Delta v=u^q
\qquad\text{in }\mathbb R^n,
\end{equation}
with $p,q>0$ and $n\ge3$, is a fundamental model for nonlinear
elliptic systems. Its scaling properties distinguish the critical
hyperbola $\frac1{p+1}+\frac1{q+1}=\frac{n-2}{n}$.
The classical Lane--Emden conjecture asserted that
\eqref{eq:lane-emden} admits no nontrivial nonnegative entire
classical solution in the subcritical range
$\frac1{p+1}+\frac1{q+1}>\frac{n-2}{n}$.
This question motivated substantial work on Liouville theorems
for cooperative systems: nonexistence was established in low
dimensions \cite{SerrinZou1996,Souplet2009}, while
higher-dimensional results covered progressively larger portions
of the subcritical region, including the recent Obata-type
approach of Li, Li, and Wei \cite{LiLiWei2025}.
Recently, OpenAI announced a proof of the full subcritical
nonexistence statement, together with its weighted
H\'{e}non--Lane--Emden analogue \cite{OpenAI2026}.

Our work is motivated in particular by the case
$q=1$, in which the Lane--Emden system reduces to
$\Delta^2v=v^p$. For $n\ge5$, the critical exponent
$p=(n+4)/(n-4)$ corresponds to the constant $Q$-curvature
equation. For the corresponding conformal problem on closed
positive Einstein manifolds, V\'etois \cite{Vetois2024}
established a uniqueness theorem through an Obata-type
integral identity. A crucial ingredient in this argument is
the positivity of the scalar curvature $R$ of the conformal
metric: it ensures the nonnegativity of the term $R|E|^2$,
where $E$ denotes the trace-free Ricci tensor, and thereby
allows the identity to yield rigidity. This scalar-curvature
positivity follows from the work of Gursky--Malchiodi
\cite{GurskyMalchiodi2015}, building on the positivity theorem
for the Paneitz operator developed by Gursky
\cite{Gursky1999}. Adapting this type of argument through
invariant tensors, Ma--Wu--Wu \cite{MaWuWu2025} proved a
Liouville theorem for the subcritical biharmonic equation
on complete noncompact manifolds with nonnegative Ricci
curvature, under the superharmonicity assumption.

Motivated by this connection, we seek a pointwise inequality
for $q>1$ analogous to the positivity of scalar curvature
in the biharmonic case. Our main result is the following.

\begin{theorem}\label{thm:scalar}
Let $(u,v)$ be a positive classical solution of
\eqref{eq:lane-emden}, where $n\ge3$, $p \ge q>1$,
and $2<(n-2)q<n$. Then
\begin{equation}\label{eq:scalar-goal}
|\nabla v|^2
<\frac{(n-2)q-2}{n-(n-2)q}\,u^qv
\qquad\text{in }\quad \mathbb R^n.
\end{equation}
\end{theorem}

We would like to also remark that the coefficient in \eqref{eq:scalar-goal} is optimal on
the critical hyperbola: for positive radial ground states,
the asymptotics in \cite{HulshofVanDerVorst1996} imply
\[
\lim_{r\to\infty}\frac{|\nabla v(r)|^2}{2u(r)^qv(r)}
=\frac{(n-2)q-2}{2[n-(n-2)q]}.
\]
Thus the bound is asymptotically attained at infinity,
despite being strict at every finite point.

At and beyond the critical hyperbola, positive entire
solutions exist \cite{HulshofVanDerVorst1996,SerrinZou1998},
and quantitative pointwise estimates remain relevant.
Early work established existence and a priori bounds for
cooperative elliptic systems
\cite{ClementDeFigueiredoMitidieri1992,Souto1995,Zou2002},
while Rellich-type identities \cite{Mitidieri1993} and
Liouville theorems
\cite{DeFigueiredoFelmer1994,Mitidieri1996,
BirindelliMitidieri1998,BuscaManasevich2002}
provided tools for nonexistence results. Further approaches include moving-sphere and monotonicity
methods \cite{ReichelZou2000,DeFigueiredoSirakov2005},
integral formulations \cite{MaChen2006,ChenLi2009},
and dynamical systems methods for radial solutions
\cite{BidautVeronGiacomini2010}.
Cheng and Huang \cite{ChengHuang2019} related subcritical
nonexistence to an energy-type a priori estimate.
Extensions include weighted systems \cite{LiZhang2019},
half-space problems \cite{LiSouplet2025}, and
Hartree-type systems \cite{DuLiYang2025}.

For pointwise control, Souplet \cite{Souplet2009}
established the component comparison
$\frac{v^{p+1}}{p+1}\le\frac{u^{q+1}}{q+1}$
for $p\ge q\ge1$, with related weighted comparisons
in \cite{PhanSouplet2012}.
Liouville theorems and rescaling yield universal growth
and decay estimates \cite{PolacikQuittnerSouplet2007},
while local Harnack-type estimates for planar systems
with $pq=1$ were obtained in \cite{LeiteMontenegro2025}.
Differential estimates for biharmonic equations were
developed in \cite{FazlyWeiXu2015,NgoNguyenPhan2018,
MaWuWu2025}.
Theorem~\ref{thm:scalar} provides such a pointwise
gradient estimate for the coupled system with $q>1$.

The paper is organized as follows.
Section~\ref{sec:basic} proves a weighted Bernstein estimate.
In Section~\ref{sec:scalar}, we present a motivating example,
derive the differential identity, and choose the coefficients
to obtain the differential inequality used in the proof.
Finally, Section~\ref{sec:main} proves Theorem~\ref{thm:scalar}
by a continuity argument and establishes the optimality of
the coefficient on the critical hyperbola.

\textbf{AI disclosure.} The authors used AI tools to help organize their original proofs
and improve the English exposition. AI tools also assisted
in computing the optimal constant in Section~\ref{subsec:final-identity}.
The authors have reviewed and verified all mathematical arguments
and take full responsibility for the content of the manuscript.

\section*{Acknowledgements}
Juncheng Wei acknowledges funding from China's National
R\&D Program (Grant No. 2022YF\\A1005602) and Hong Kong General
Research Fund project (No. 14303125)
``On Fujita equation in the critical or supercritical regime''.

\section{Some basic estimates}\label{sec:basic}

Let $\eta_R\in C_c^\infty(B_{2R})$ such that
\begin{equation}\label{eq:cutoff}
 0\le\eta_R\le1,\quad \eta_R=1\text{ in }B_R,
 \quad |\nabla\eta_R|\le c_n/R,\quad |\nabla^2 \eta_R|\le c_n/R^2.
\end{equation}

We recall the following lemma from \cite{MaWuWu2025}, which provides
some motivation for the Bernstein technique used in this paper.

\begin{lemma}\cite[Lemma~2.1]{MaWuWu2025}.\label{lem:original-bernstein}
Let $w\in C^4(\mathbb R^n)$ be positive, and let $a\ge0$. Define
\[
Z=\frac{\Delta w}{w}+a\frac{|\nabla w|^2}{w^2}.
\]
Suppose that there exists $\varepsilon_0>0$, independent of $R$,
such that
\[
\operatorname{div}(w^{2-2a}\nabla Z)
\ge \varepsilon_0 w^{2-2a}
\left(Z^2+\frac{|\nabla w|^4}{w^4}\right)
\]
at every positive maximum point of $\eta_R^4Z$, for all
sufficiently large $R$. Then $Z\le0$ in $\mathbb R^n$.
\end{lemma}

We shall need the following weighted version of the
Bernstein estimate, which allows a more general drift term.

\begin{lemma}\label{lem:weighted-bernstein}
Let $u,v\in C^2(\mathbb R^n)$ be positive, and let
$W\in C^2(\mathbb R^n)$. Suppose that $\sigma\ge0$ satisfies
\begin{equation}\label{eq:weight-growth}
\lim_{R\to\infty}R^{-2}\sup_{B_{2R}}u^{-\sigma}=0.
\end{equation}
Assume that there exist constants $C_1, C_2>0$, $\beta\in\mathbb R$,
and a continuous vector field $X$ such that, on $\{W>0\}$,
\begin{align}
0<W&\le C_1 u^{-\sigma}\frac{|\nabla v|^2}{v^2},
\label{eq:weighted-upper}\\
\Delta W+
\left(X+\beta\frac{\nabla v}{v}\right)\cdot\nabla W
&\ge C_2\left(
u^{-\sigma}\frac{|\nabla v|^4}{v^4}
+W|X|^2\right).
\label{eq:weighted-bernstein-hyp}
\end{align}
Then $W\le0$ in $\mathbb R^n$.
\end{lemma}

\begin{proof}

Suppose that $\eta^4W$ attains a positive maximum at $x_R$.
Then $\eta(x_R)>0$, $W(x_R)>0$, and
\[
\nabla W=-4W\frac{\nabla\eta}{\eta},
\qquad
\Delta(\eta^4W)\le0
\quad\text{at }x_R.
\]
All quantities in the following estimates are evaluated at $x_R$. Expanding the latter inequality and substituting the identity
for $\nabla W$, we obtain
\[
\eta^4\Delta W
\le -4W\eta^3\Delta\eta
+20W\eta^2|\nabla\eta|^2.
\]
Thus, the differential inequality for $W$ gives
\begin{align*}
&C_2u^{-\sigma}\eta^4\frac{|\nabla v|^4}{v^4}
+C_2W\eta^4|X|^2\\
&\le
-4W\eta^3\Delta\eta
+20W\eta^2|\nabla\eta|^2
-4W\eta^3
\left(X+\beta\frac{\nabla v}{v}\right)
\cdot\nabla\eta.
\end{align*}

By Young's inequality,
\[
C_2W\eta^4|X|^2
+4W\eta^3X\cdot\nabla\eta
\ge -\frac4{C_2}W\eta^2|\nabla\eta|^2.
\]
we obtain
\[
u^{-\sigma}\eta^4\frac{|\nabla v|^4}{v^4}
\le
C'W\left(
\frac{\eta^2}{R^2}
+\frac{\eta^3}{R}\frac{|\nabla v|}{v}
\right).
\]
On the other hand, \eqref{eq:weighted-upper} gives
\[
0<W\le C_1u^{-\sigma}\frac{|\nabla v|^2}{v^2},
\]
so $|\nabla v|>0$ at $x_R$. Substituting this upper bound
for $W$, we get
$u^{-\sigma}\eta^2|\nabla v|^2/v^2$, we obtain
\[
\eta^2\frac{|\nabla v|^2}{v^2}
\le \frac{C'}{R^2}
+\frac{C'\eta}{R}\frac{|\nabla v|}{v}.
\]
Another application of Young's inequality gives
\[
\eta^2\frac{|\nabla v|^2}{v^2}
\le \frac{C'}{R^2}.
\]
Using \eqref{eq:weighted-upper} once more, we conclude that
\[
\sup_{B_R}W
\le (\eta^4W)(x_R)
\le C_1u(x_R)^{-\sigma}\eta(x_R)^2
\left(\eta^2\frac{|\nabla v|^2}{v^2}\right)(x_R)
\le \frac{C'}{R^2}\sup_{B_{2R}}u^{-\sigma}.
\]
Letting $R\to\infty$ and using
\eqref{eq:weight-growth}, we obtain $W\le0$ throughout
$\mathbb R^n$, completing the proof.
\end{proof}

Finally, we record an elementary lower bound for positive
superharmonic functions.
\begin{lemma}
\label{lem:lower-barrier}
If $u>0$ and $-\Delta u\ge0$ in $\mathbb R^n$, $n\ge3$, then $
 u(z)\ge c(1+|z|)^{2-n}.
$

\end{lemma}
\begin{proof}
For $R>1$, compare $u$ on $B_R\setminus\overline{B_1}$
with
\[
 \min_{\partial B_1}u\cdot \,
 \bigg(\frac{|z|^{2-n}-R^{2-n}}{1-R^{2-n}}\bigg).
\]
Letting $R\to\infty$
gives the asserted lower bound.
\end{proof}
In particular, we remark that \eqref{eq:weight-growth} holds for
$0\le\sigma<2/(n-2)$.

\section{Derivation of the differential inequality}\label{sec:scalar}

\subsection{A motivating example:}

We begin with a simple weighted differential inequality that
illustrates the cancellation used in the general argument.
Consider the auxiliary function
\[
F_a:=\frac{|\nabla v|^2}{2u^q}-C_a v.
\]
The weight $u^{2q-4a}$ allows us to adjust the mixed Hessian term,
while the constant $C_a$ is chosen to cancel the remaining
multiple of $\langle\nabla u,\nabla v\rangle$.

\begin{proposition}
\label{prop:b-zero}
For
$
\frac q2<a<\frac{nq}{4},
$
define
\[
C_a:=\frac{q-4a/n}{4a-2q}>0,
\qquad
A_q(a):=
\frac{q(q+1)}2+q(2a-q)-\frac{4(n-1)}n a^2.
\]
Then, we have the following inequality:
\begin{align}
\operatorname{div}\left(
u^{2q-4a}\nabla F_a
\right)
\ge{}&
A_q(a)\,
u^{q-4a-2}|\nabla u|^2|\nabla v|^2
\notag\\
&+\left(C_a+\frac1n\right)u^{3q-4a}
+\frac q2\,u^{q-4a-1}v^p|\nabla v|^2.
\label{eq:b-zero-inequality}
\end{align}
\end{proposition}

\begin{proof}
A straightforward computation yields:
\begin{align*}
\operatorname{div}\left(u^{2q-4a}\nabla F_a\right)
={}&u^{q-4a}
\left(|D^2v|^2-\frac{4a}{u}D^2v(\nabla u,\nabla v)\right)\\
&+\left[\frac{q(q+1)}2+q(2a-q)\right]
u^{q-4a-2}|\nabla u|^2|\nabla v|^2\\
&+\left[C_a(4a-2q)-q\right]
u^{2q-4a-1}\langle\nabla u,\nabla v\rangle\\
&+C_a u^{3q-4a}
+\frac q2\,u^{q-4a-1}v^p|\nabla v|^2.
\end{align*}

To estimate the Hessian terms, introduce the trace-free tensor
\[
T:=\frac{\nabla u\otimes\nabla v+\nabla v\otimes\nabla u}{2}
-\frac{\langle\nabla u,\nabla v\rangle}{n}I.
\]
Since $\Delta v=-u^q$, completing the square yields
\begin{align*}
|D^2v|^2-\frac{4a}{u}D^2v(\nabla u,\nabla v)
={}&
\left|D^2v+\frac{u^q}{n}I-\frac{2a}{u}T\right|^2
-\frac{4a^2}{u^2}|T|^2\\
&+\frac{u^{2q}}n
+\frac{4a}{n}u^{q-1}\langle\nabla u,\nabla v\rangle.
\end{align*}
Moreover,
\[
|T|^2
=\frac12|\nabla u|^2|\nabla v|^2
+\frac{n-2}{2n}\langle\nabla u,\nabla v\rangle^2
\le\frac{n-1}{n}|\nabla u|^2|\nabla v|^2.
\]
Consequently,
\begin{align*}
|D^2v|^2-\frac{4a}{u}D^2v(\nabla u,\nabla v)
\ge{}&
\frac{u^{2q}}n
+\frac{4a}{n}u^{q-1}\langle\nabla u,\nabla v\rangle\\
&-\frac{4(n-1)a^2}{nu^2}|\nabla u|^2|\nabla v|^2.
\end{align*}
Substituting this estimate into the divergence identity proves
\eqref{eq:b-zero-inequality}, since
\[
C_a(4a-2q)-q+\frac{4a}{n}=0.
\]
\end{proof}

The sign of the remaining gradient coefficient explains the
restriction on $q$ in this calculation. Indeed,
\[
A_q\left(\frac q2\right)
=\frac{q}{2n}\bigl(n-(n-2)q\bigr),
\]
and $A_q(a)$ is strictly decreasing for $a>q/2$.
Thus, if $
0<q<\frac{n}{n-2},
$
we may choose $a>q/2$ sufficiently close to $q/2$ so that
$A_q(a)>0$. For this choice, we have:
\[
\operatorname{div}\left(u^{2q-4a}\nabla F_a\right)
\ge
\left(C_a+\frac1n\right)u^{3q-4a}>0.
\]

\subsection{The differential identity}\label{subsec:original-identity}

We first record an algebraic identity involving the Hessian that underlies the computations below. For arbitrary $a,b\in\mathbb R$, define
\begin{equation}\label{eq:original-T}
T_{ij}:=v_{ij}
-a\frac{u_iv_j+u_jv_i}{u}
-b\frac{v_iv_j}{v},
\qquad
\mathring{T}:=T-\frac{\operatorname{tr}T}{n}I.
\end{equation}

\begin{lemma}[Obata-type identity]\label{lem:obata}
For $-\Delta v=u^q$, the following identity holds pointwise:
\begin{align}
&|\nabla^2v|^2
-\frac{4a}{u}v_{ij}u_iv_j
-\frac{2b}{v}v_{ij}v_iv_j
\notag\\
={}&|\mathring{T}|^2
+\frac{2(n-2)a^2}{nu^2}
\bigl(|\nabla u|^2|\nabla v|^2
      -\langle\nabla u,\nabla v\rangle^2\bigr)
\notag\\
&-\frac{4(n-1)a^2}{nu^2}|\nabla u|^2|\nabla v|^2
-\frac{4(n-1)ab}{nuv}|\nabla v|^2
 \langle\nabla u,\nabla v\rangle
-\frac{n-1}{n}b^2\frac{|\nabla v|^4}{v^2}
\notag\\
&+\frac1n u^{2q}
+\frac{4a}{n}u^{q-1}\langle\nabla u,\nabla v\rangle
+\frac{2b}{n}\frac{u^q|\nabla v|^2}{v}.
\label{eq:obata}
\end{align}
\end{lemma}

\begin{proof}
Writing
\[
\operatorname{tr}T
=-u^q-\frac{2a}{u}\langle\nabla u,\nabla v\rangle
-\frac b v|\nabla v|^2,
\]
use $|T|^2=|\mathring{T}|^2+(\operatorname{tr}T)^2/n$ in
\[
\begin{aligned}
|\nabla^2v|^2
-\frac{4a}{u}v_{ij}u_iv_j
-\frac{2b}{v}v_{ij}v_iv_j
=|T|^2
-\left|a\frac{\nabla u\otimes\nabla v+
\nabla v\otimes\nabla u}{u}
+b\frac{\nabla v\otimes\nabla v}{v}\right|^2.
\end{aligned}
\]
Expanding the two squares gives \eqref{eq:obata}.
\end{proof}

For $C>0$ and real parameters $s,t,x,y$, let
\begin{equation}\label{eq:original-H}
H=u^{-2s}v^{2t}
\left(\frac12\frac{|\nabla v|^2}{v^2}
-C\frac{u^q}{v}\right).
\end{equation}
We retain all six parameters $a,b,s,t,x,y$ as independent
in the next identity. 

\begin{lemma}\label{lem:general-identity}
For arbitrary $a,b,s,t,x,y\in\mathbb R$ and $C>0$,
\begin{align}
&u^{2s-x}v^{2-2t-y}
\operatorname{div}(u^xv^y\nabla H)
\notag\\
={}&|\mathring{T}|^2
+\frac{2(n-2)a^2}{nu^2}
 \bigl(|\nabla u|^2|\nabla v|^2
 -\langle\nabla u,\nabla v\rangle^2\bigr)
\notag\\
&+(4a-4s+x)\frac{v_{ij}u_iv_j}{u}
+(2b+4t-4+y)\frac{v_{ij}v_iv_j}{v}
\notag\\
&+\left[-\frac{4(n-1)}n ab+(t-1)x-sy-2s(2t-2)\right]
 \frac{|\nabla v|^2\langle\nabla u,\nabla v\rangle}{uv}
\notag\\
&+\left[\frac{4a}{n}-C(q-2s)y-C(2t-1)x
-q-2C(q-2s)(2t-1)\right]
 u^{q-1}\langle\nabla u,\nabla v\rangle
\notag\\
&+\left[-\frac{4(n-1)}n a^2+s(2s+1-x)\right]
 \frac{|\nabla u|^2|\nabla v|^2}{u^2}
\notag\\
&+\left[-\frac{n-1}{n}b^2+(1-t)(3-2t-y)\right]
 \frac{|\nabla v|^4}{v^2}
\notag\\
&+\left[\frac{2b}{n}-(t-1)
 -C(2t-1)(2t-2+y)\right]
 \frac{u^q|\nabla v|^2}{v}
\notag\\
&-C(q-2s)(q-2s-1+x)u^{q-2}v|\nabla u|^2
\notag\\
&+\left[\frac1n+C(2t-1)\right]u^{2q}
+2su^{2s}v^{2-2t}H\frac{v^p}{u}
+Cq u^{q-1}v^{p+1}.
\label{eq:general-identity}
\end{align}
\end{lemma}

\begin{proof}
By straightforward computation:
\begin{align*}
&\Delta\left(\frac12|\nabla v|^2-Cu^qv\right)
=|\nabla^2v|^2
-q(1+2C)u^{q-1}\langle\nabla u,\nabla v\rangle
+C u^{2q}\\
&\hspace{27mm}
+Cq u^{q-1}v^{p+1}
-Cq(q-1)u^{q-2}v|\nabla u|^2,\\
&\nabla\left(\frac12|\nabla v|^2-Cu^qv\right)
=\nabla^2v(\nabla v,\cdot)
-Cqu^{q-1}v\nabla u-Cu^q\nabla v.
\end{align*}
Applying the product rule to \eqref{eq:original-H}, we get
\begin{align*}
&u^{2s-x}v^{2-2t-y}
\operatorname{div}(u^xv^y\nabla H)\\
={}&\Delta\left(\frac12|\nabla v|^2-Cu^qv\right)\\
&+\left[(x-4s)\frac{\nabla u}{u}
+(y+4t-4)\frac{\nabla v}{v}\right]
 \cdot\nabla\left(\frac12|\nabla v|^2-Cu^qv\right)\\
&+\left[2s\frac{v^p}{u}-(2t-2)\frac{u^q}{v}
+2s(2s+1-x)\frac{|\nabla u|^2}{u^2}\right.\\
&\left.\hspace{16mm}
+(2t-2)(2t-3+y)\frac{|\nabla v|^2}{v^2}
+\bigl((2t-2)x-2sy-4s(2t-2)\bigr)
 \frac{\langle\nabla u,\nabla v\rangle}{uv}\right]\\
&\qquad\times\left(\frac12|\nabla v|^2-Cu^qv\right).
\end{align*}
We then apply Lemma~\ref{lem:obata} to the Hessian terms and
collect the coefficients. This proves \eqref{eq:general-identity}.
\end{proof}

\subsection{Choice of the coefficients}\label{subsec:parameters}

We eliminate the two extra Hessian terms and the two mixed-gradient
terms in Lemma~\ref{lem:general-identity}. Their coefficients give
\begin{equation}\label{eq:four-cancellations}
\begin{cases}
4a-4s+x=0,\\
2b+4t-4+y=0,\\
-\dfrac{4(n-1)}n ab+(t-1)x-sy-2s(2t-2)=0,\\
\dfrac{4a}{n}-C(q-2s)y-C(2t-1)x-q-2C(q-2s)(2t-1)=0.
\end{cases}
\end{equation}
We note that the term $4a/n$ in the last equation comes from the trace of
$T$ in Lemma~\ref{lem:obata}.

Set $t=1$. The first three equations reduce to
\[
x=4s-4a,\qquad y=-2b,\qquad
\left(2s-\frac{4(n-1)}n a\right)b=0.
\]
We work with the branch $b\ne0$, obtaining
\begin{equation}\label{eq:sxy}
s=\frac{2(n-1)}n a,\qquad x=4s-4a,\qquad y=-2b.
\end{equation}
Substitution into the fourth coefficient gives exactly
\[
\frac{4a}{n}-q
+C\bigl[2b(q-2s)-x-2(q-2s)\bigr].
\]
It vanishes if and only if
\begin{equation}\label{eq:b-cancellation}
C\left(4a-2q+2bq-\frac{8(n-1)}n ab\right)
=q-\frac{4a}{n},
\end{equation}
so that
\begin{equation}\label{eq:b-choice}
b=\frac{q-2a+\frac1{2C}(q-4a/n)}{q-2s}.
\end{equation}

The remaining two terms involving $|\nabla u|^2$ are collected into a multiple of $H$. Put
\begin{equation}\label{eq:k-choice}
k:=2s\left(1-\frac x2\right).
\end{equation}
The relations \eqref{eq:sxy} yield the exact identities

\begin{align}
-\frac{4(n-1)}n a^2+s(2s+1-x)&=\frac k2,\notag\\
(q-2s)(q-2s-1+x)-k
&=q(q-1)+\frac{8(n-1)}n a^2-4aq
\label{eq:coefficient-computation}
\end{align}

so that, using $u^{2s}H=\frac12|\nabla v|^2-Cu^qv$, we have:
\begin{align}
&u^{-2s}\biggl[
\left(-\frac{4(n-1)}n a^2+s(2s+1-x)\right)
\frac{|\nabla u|^2|\nabla v|^2}{u^2}
-C(q-2s)(q-2s-1+x)u^{q-2}v|\nabla u|^2
\biggr]\notag\\
={}&k H\frac{|\nabla u|^2}{u^2}
-C\left(q(q-1)+\frac{8(n-1)}n a^2-4aq\right)
 u^{q-2s-2}v|\nabla u|^2.
\label{eq:H-gradient-computation}
\end{align}
This computation explains the final condition we impose:
\begin{equation}\label{eq:a-constraint}
q(q-1)+\frac{8(n-1)}n a^2-4aq=0,
\end{equation}
which reduces \eqref{eq:H-gradient-computation} to
$kH|\nabla u|^2/u^2$.

We choose the smaller root of \eqref{eq:a-constraint}:
\begin{equation}\label{eq:a-choice}
2s=q-\sqrt{\frac{q[2(n-1)-(n-2)q]}n},
\qquad a=\frac{ns}{2(n-1)},
\qquad x=\frac{2(n-2)s}{n-1}.
\end{equation}
Indeed,
\[
(q-2s)^2-1=\frac{(q-1)[n-(n-2)q]}n>0,
\qquad 0<2s<q-1<\frac2{n-2}.
\]
This ensures the growth condition of Lemma~\ref{lem:weighted-bernstein}. Moreover,
\[
k=(q-2s)(q-2s-1+x)>0.
\]
In particular, \eqref{eq:b-choice} gives $b>0$, as required for
our choice of branch. The coefficients $a,s,x$ depend only on $n,q$ while
$b,y$ also depend on $C$.

\subsection{The inequalities used in the proof}\label{subsec:final-identity}

For every $\lambda>0$, we make the following choices from
Section~\ref{subsec:parameters}:
\begin{equation}\label{eq:chosen-parameters}
\begin{cases}
C=\dfrac1{2\lambda},\qquad t=1,\\[3pt]
s=\dfrac12\left(q-\sqrt{\dfrac{q[2(n-1)-(n-2)q]}{n}}\right),\\[6pt]
a=\dfrac{ns}{2(n-1)},\qquad x=4s-4a=\dfrac{2(n-2)s}{n-1},\\[6pt]
b=\dfrac{q-2a+\lambda(q-4a/n)}{q-2s},\qquad y=-2b.
\end{cases}
\end{equation}
In particular, $0<2s<2/(n-2)$, $b>0$, and
$k=2s(1-x/2)>0$. We first record the identity obtained by
substituting \eqref{eq:chosen-parameters} into
Lemma~\ref{lem:general-identity}, retaining the original $H$.

\begin{corollary}\label{cor:H-identity}
With the coefficients in \eqref{eq:chosen-parameters},
the function $H$ in \eqref{eq:original-H} satisfies
\begin{align}
&u^{-x}v^{2b}\operatorname{div}(u^xv^{-2b}\nabla H)
\notag\\
={}&u^{-2s}|\mathring{T}|^2
+\frac{2(n-2)a^2}{n}u^{-2s-2}
\bigl(|\nabla u|^2|\nabla v|^2
-\langle\nabla u,\nabla v\rangle^2\bigr)
\notag\\
&+k H\frac{|\nabla u|^2}{u^2}
+u^{-2s}\left[
-\frac{n-1}{n}b^2\frac{|\nabla v|^4}{v^2}
+\left(C+\frac1n\right)
\left(2b\frac{u^q|\nabla v|^2}{v}+u^{2q}\right)
\right]
\notag\\
&+2sH\frac{v^p}{u}
+qC u^{q-2s-1}v^{p+1},
\label{eq:H-identity}
\end{align}
\end{corollary}

\begin{proof}
Substitute the coefficients from Section~\ref{subsec:parameters}
into Lemma~\ref{lem:general-identity}. All four mixed
coefficients vanish by \eqref{eq:four-cancellations}, and
\eqref{eq:H-gradient-computation}--\eqref{eq:a-constraint}
combine the $|\nabla u|^2$ terms into $kH|\nabla u|^2/u^2$.
The remaining terms give \eqref{eq:H-identity}.
\end{proof}

The identity \eqref{eq:H-identity} already exhibits nice cancellations and some positivity estimates,
but the cutoff argument is more naturally applied after dividing
by $v^2$. On the set  $\{H >0 \}$, its immediate upper bound
involves $u^{-2s}|\nabla v|^2$, whereas the weighted Bernstein type lemma \ref{lem:weighted-bernstein} is formulated using $u^{-2s}|\nabla v|^2/v^2$.

We therefore consider
\begin{equation}\label{eq:Z-choice}
Z_\lambda:=\frac{2\lambda}{v^2}H
=u^{-2s}\left(\lambda\frac{|\nabla v|^2}{v^2}
-\frac{u^q}{v}\right).
\end{equation}
The resulting drift contains $x\nabla u/u$, for which no
independent pointwise bound is available. Although
\eqref{eq:H-identity} contains no explicit mixed-gradient term,
the product rule for $H=v^2Z_\lambda/(2\lambda)$ introduces
$-2xZ_\lambda\langle\nabla u,\nabla v\rangle/(uv)$.
Completing the square using $k Z_\lambda|\nabla u|^2/u^2$
yields a nonnegative term that controls this drift at a positive
maximum of the cutoff-modified quantity. Such a term naturally
arises when modifying the drift of the elliptic operator.

For the remaining terms, set
\begin{align}
p_\lambda
&:=-\frac{2(n-1)}n b^2+2b
+\lambda\left(1+\frac{4b}{n}\right)
+\frac{2\lambda^2}{n}\notag\\
&=\frac{(1+2\lambda)
[n-(n-2)q-\lambda((n-2)q-2)]}
{2n-2-(n-2)q}.
\label{eq:p-formula}
\end{align}
The following is the form used in the proof of the pointwise estimate.

\begin{corollary}[The sub-elliptic differential inequality]
\label{cor:final-identity}
For every $\lambda>0$, the function $Z_\lambda$ satisfies
\begin{align}
&\Delta Z_\lambda
+\left(x\frac{\nabla u}{u}
+(4-2b)\frac{\nabla v}{v}\right)\cdot\nabla Z_\lambda
\notag\\
={}&\frac{2\lambda u^{-2s}}{v^2}|\mathring{T}|^2
+\frac{4\lambda(n-2)a^2}{n}u^{-2s-2}
\frac{|\nabla u|^2|\nabla v|^2
-\langle\nabla u,\nabla v\rangle^2}{v^2}
\notag\\
&+k Z_\lambda
\left|\frac{\nabla u}{u}-\frac{x}{k}\frac{\nabla v}{v}\right|^2
+\lambda p_\lambda u^{-2s}\frac{|\nabla v|^4}{v^4}
\notag\\
&+\left[2(b-1)-\frac{x^2}{k}
-\frac{4\lambda}{n}(b+\lambda)\right]
 Z_\lambda\frac{|\nabla v|^2}{v^2}
+\left(\frac{2\lambda}{n}-1\right)u^{2s}Z_\lambda^2
\notag\\
&+2sZ_\lambda\frac{v^p}{u}
+q u^{q-2s-1}v^{p-1}.
\label{eq:Z-identity}
\end{align}
\end{corollary}

\begin{proof}
Since $H=v^2Z_\lambda/(2\lambda)$, the product rule gives
\begin{align}
&\frac{2\lambda}{v^2}
 u^{-x}v^{2b}\operatorname{div}(u^xv^{-2b}\nabla H)
\notag\\
={}&\Delta Z_\lambda
+\left(x\frac{\nabla u}{u}
+(4-2b)\frac{\nabla v}{v}\right)\cdot\nabla Z_\lambda
\notag\\
&+\left[-2\frac{u^q}{v}
+(2-4b)\frac{|\nabla v|^2}{v^2}
+2x\frac{\langle\nabla u,\nabla v\rangle}{uv}\right]Z_\lambda.
\label{eq:conversion}
\end{align}
Insert Corollary~\ref{cor:H-identity} into
\eqref{eq:conversion} and substitute
\[
\frac{u^q}{v}
=\lambda\frac{|\nabla v|^2}{v^2}-u^{2s}Z_\lambda.
\]
The quartic terms combine into
$\lambda p_\lambda u^{-2s}|\nabla v|^4/v^4$ by
\eqref{eq:p-formula}. The terms proportional to $Z_\lambda$ containing
$\nabla u$ are
\begin{align*}
&k Z_\lambda\frac{|\nabla u|^2}{u^2}
-2xZ_\lambda\frac{\langle\nabla u,\nabla v\rangle}{uv}\\
&\quad=k Z_\lambda
\left|\frac{\nabla u}{u}-\frac{x}{k}\frac{\nabla v}{v}\right|^2
-\frac{x^2}{k}Z_\lambda\frac{|\nabla v|^2}{v^2}.
\end{align*}
Here $k>0$ by \eqref{eq:k-choice}. Collecting the remaining
terms proves \eqref{eq:Z-identity}.
\end{proof}

\section{Proof of Theorem~\ref{thm:scalar}}\label{sec:main}

\begin{proof}[Proof of Theorem~\ref{thm:scalar}]
Set
\[
\lambda_*:=\frac{n-(n-2)q}{(n-2)q-2},
\qquad
A_\lambda:=2(b-1)-\frac{x^2}{k}
-\frac{4\lambda}{n}(b+\lambda).
\]
By \eqref{eq:p-formula}, $p_\lambda>0$ for
$0<\lambda<\lambda_*$ and $p_{\lambda_*}=0$.

\medskip
\noindent\textbf{Step 1:} $Z_\lambda \le 0$ for sufficiently small $\lambda$.

As $\lambda\downarrow0$, we have
$b\to1+x/[2(q-2s)]$, and hence
\[
A_\lambda\longrightarrow
\frac{x(q-2s-1)}{(q-2s)(q-2s-1+x)}>0.
\]
Thus $A_\lambda\ge\lambda$ for all sufficiently small
$\lambda>0$. On $\{Z_\lambda>0\}$,
\[
0<u^{2s}Z_\lambda
\le\lambda\frac{|\nabla v|^2}{v^2}.
\]
Since $2\lambda/n-1\ge-1$, it follows that
\[
A_\lambda Z_\lambda\frac{|\nabla v|^2}{v^2}
+\left(\frac{2\lambda}{n}-1\right)u^{2s}Z_\lambda^2
\ge0.
\]
Set
\[
X:=x\left(\frac{\nabla u}{u}
-\frac{x}{k}\frac{\nabla v}{v}\right),
\qquad
\beta_\lambda:=4-2b+\frac{x^2}{k}.
\]
Discarding the remaining non-negative terms in
\eqref{eq:Z-identity}, we obtain, on $\{Z_\lambda>0\}$,
\[
\Delta Z_\lambda+
\left(X+\beta_\lambda\frac{\nabla v}{v}\right)
\cdot\nabla Z_\lambda
\ge
\lambda p_\lambda u^{-2s}\frac{|\nabla v|^4}{v^4}
+\frac{k}{x^2}Z_\lambda|X|^2.
\]
Lemma~\ref{lem:weighted-bernstein} therefore gives
$Z_\lambda\le0$ for all sufficiently small $\lambda>0$.
We again note that the required growth condition follows from
Lemma~\ref{lem:lower-barrier} and $2s<2/(n-2)$.

\medskip
\noindent\textbf{Step 2:} The continuity method.

Suppose that $Z_{\lambda_0}\le0$ for some
$0<\lambda_0<\lambda_*$. For $\lambda=\lambda_0+h$,
on $\{Z_\lambda>0\}$ we have
\[
0<u^{2s}Z_\lambda
=u^{2s}Z_{\lambda_0}
+h\frac{|\nabla v|^2}{v^2}
\le h\frac{|\nabla v|^2}{v^2}.
\]
Since $A_\lambda$ and $2\lambda/n-1$ are continuous in
$\lambda$, we may choose $M>0$ such that
\[
|A_\lambda|\le M,
\qquad
\left|\frac{2\lambda}{n}-1\right|\le M
\]
for $\lambda$ near $\lambda_0$. On $\{Z_\lambda>0\}$,
the bound
$Z_\lambda\le h u^{-2s}|\nabla v|^2/v^2$ gives
\begin{align*}
A_\lambda Z_\lambda\frac{|\nabla v|^2}{v^2}
&\ge -MZ_\lambda\frac{|\nabla v|^2}{v^2}
\ge -Mh\,u^{-2s}\frac{|\nabla v|^4}{v^4},\\
\left(\frac{2\lambda}{n}-1\right)u^{2s}Z_\lambda^2
&\ge -Mu^{2s}Z_\lambda^2
\ge -Mh^2u^{-2s}\frac{|\nabla v|^4}{v^4}.
\end{align*}
Substituting these estimates into \eqref{eq:Z-identity}
and discarding the remaining non-negative terms yields
\[
\Delta Z_\lambda+
\left(X+\beta_\lambda\frac{\nabla v}{v}\right)
\cdot\nabla Z_\lambda
\ge
\bigl[\lambda p_\lambda-M(h+h^2)\bigr]
u^{-2s}\frac{|\nabla v|^4}{v^4}
+\frac{k}{x^2}Z_\lambda|X|^2.
\]
As $h\downarrow0$, the coefficient in brackets converges to
$\lambda_0p_{\lambda_0}>0$, and is therefore positive for
all sufficiently small $h>0$. Together with
\[
0<Z_\lambda\le h u^{-2s}\frac{|\nabla v|^2}{v^2}
\qquad\text{on }\{Z_\lambda>0\},
\]
this verifies the hypotheses of
Lemma \ref{lem:weighted-bernstein}, with the same growth
condition as in Step 1. Hence $Z_{\lambda_0+h}\le0$.

By pointwise continuity and monotonicity in $\lambda$, the set
$
\left\{\lambda\in[0,\lambda_*]:
Z_\lambda\le0\text{ in }\mathbb R^n\right\}
$
is a closed interval. Step 1 shows that this interval has
positive length, while Step 2 proves that upper endpoint is exactly $\lambda_*$. Consequently, $Z_{\lambda_*}\le0$, or
equivalently,
\[
\frac12|\nabla v|^2
\le\frac{(n-2)q-2}{2[n-(n-2)q]}u^qv.
\]

If equality held at some point, then at that point we would have
$Z_\lambda=0$, $\nabla Z_\lambda=0$, and $\Delta Z_\lambda\le0$.
The identity \eqref{eq:Z-identity}, with $p_\lambda=0$, would then yield
\[
0\ge\Delta Z_\lambda
\ge q u^{q-2s-1}v^{p-1}>0,
\]
a contradiction. Hence the estimate is strict.
\end{proof}

\begin{remark}[Sharpness on the critical hyperbola]\label{rem:sharp-Cq}
Suppose in addition that
\[
\frac{1}{p+1}+\frac{1}{q+1}=\frac{n-2}{n}.
\]
Let $(u,v)$ be a positive radial ground state of the system.
Such solutions exist and satisfy the asymptotic behavior
(see \cite{HulshofVanDerVorst1996})
\[
u(r)=A r^{2-n}(1+o(1))
\qquad\text{as }r\to\infty,
\]
for some $A>0$. Integrating the radial equation
$-\Delta v=u^q$ and using $(n-2)q<n$, we find
\begin{align*}
-v'(r)
&=r^{1-n}\int_0^r \rho^{n-1}u(\rho)^q\,d\rho\\
&\sim\frac{A^q}{n-(n-2)q}\,r^{1-(n-2)q}.
\end{align*}
Since $v(r)\to0$ and $(n-2)q>2$, integration from
$r$ to infinity gives
\[
v(r)\sim
\frac{A^q}{[(n-2)q-2][n-(n-2)q]}
\,r^{2-(n-2)q}.
\]
Consequently,
\[
\lim_{r\to\infty}\frac{|\nabla v(r)|^2}{2u(r)^qv(r)}
=\frac{(n-2)q-2}{2[n-(n-2)q]}.
\]
Thus the coefficient in Theorem~\ref{thm:scalar} is optimal
among positive entire solutions on the critical hyperbola,
even though the inequality is strict at every finite point.
\end{remark}

\bibliographystyle{abbrv}
\bibliography{lane}

\end{document}